\documentclass[a4paper, UKenglish]{article}
\usepackage[UKenglish]{babel}
\usepackage{amsmath,amsthm,amssymb,amsfonts}
\usepackage{tikz}
\usepackage{tikz-cd}
\usetikzlibrary{arrows}
\usepackage{ stmaryrd }
\usepackage{extarrows}
\usepackage{enumerate}
\usepackage{enumitem}
\usepackage{hyperref}
\usepackage{url}

\newcommand{\Z}{\mathbb{Z}}
\newcommand{\Q}{\mathbb{Q}}
\newcommand{\R}{\mathbb{R}}
\newcommand{\C}{\mathbb{C}}

\newtheorem{theorem}{Theorem}[section]

\newtheorem{lemma}[theorem]{Lemma}

\theoremstyle{definition}
\newtheorem{definition}[theorem]{Definition}

\newtheorem{remark}[theorem]{Remark}

\title{ Integral points on
affine quadric surfaces}
\author{Tim Santens}
\date{}
\begin{document}
\maketitle
\begin{abstract}
 It is well-known that the Hasse principle holds for quadric hypersurfaces. The Hasse principle fails for integral points on smooth quadric hypersurfaces of dimension 2 but the failure can be completely explained by the Brauer-Manin obstruction. We investigate how often the family of quadric hypersurfaces $ax^2 + by^2 +cz^2 = n$ has a Brauer-Manin obstruction. We improve previous bounds of Mitankin.
\end{abstract}
\section{Introduction}
One of the oldest questions in number theory is whether a particular  polynomial equation has a solution in the integers or in the rational numbers. The general problem of finding an algorithm which can decide for every polynomial $f \in \Z[X_1, \cdots , X_n]$ whether it has an integral zero is known as Hilbert's 10th problem. It was shown to be impossible in the second half of the twentieth century by the combined work of multiple authors. The analogous question for rational zeros is still open but is also expected to be unsolvable. A slight generalization in which we allow systems of polynomial equations can be restated in a more modern terminology as the question whether for a $\Q$-variety $X$ and an integral model $\mathcal{X}$ there are any rational points $X(\Q) = \mathcal{X}(\Q)$ or integral points $\mathcal{X}(\Z)$. Note that the second question is only interesting if $\mathcal{X}$ is not proper (and thus not projective) since if it were proper, then because of the valuative criterion of properness $\mathcal{X}(\Z) = \mathcal{X}(\Q)$. Necessary conditions for the existence of rational or integral points are that $X(\Q_p) \neq \emptyset$ for all prime numbers $p$ and $X(\R) \neq \emptyset$ , respectively that $\mathcal{X}(\Z_p) \neq \emptyset$ for all prime numbers $p$ and $\mathcal{X}(\R) \neq 0$. Here $\Q_p$ are the $p$-adic numbers and $\Z_p$ are the $p$-adic integers. If these conditions are also sufficient we say that $X$ satisfies the Hasse principle, respectively that $\mathcal{X}$ satisfies the integral Hasse principle.

The simplest example of varieties, those defined by a system of linear equations, satisfy the Hasse principle and the integral Hasse principle by linear algebra and Euclid's algorithm. The first non-trivial example are quadrics. Both projective and affine quadrics satisfy the Hasse principle  \cite[\S 4: Theorem 8]{serre2012course}, this was proved by Minkowski.
Sadly, this does not generalize to their integral models. The situation depends on the dimension of the variety. We may assume that the associated quadratic form is non-degenerate, equivalently that the corresponding variety is smooth. If its dimension is greater or equal to 3 and the set of real solutions is unbounded, then they satisfy the Hasse principle, \cite[Theorem 6.1]{colliot2009brauer}. This was originally proven by Kneser.
Note that the unboundedness assumption is necessary due to counterexamples like $4x^2 + 4y^2 + 4z^2 + 9t^2 = 1$. This theorem fails when the rank of the quadratic form is less than 4. In the case when the rank is 3 the failure of the integral Hasse principle can be completely explained by the Brauer-Manin obstruction \cite[Theorem 6.3]{colliot2009brauer} which was proven by Colliot-Thélène and Xu.
The dimension 1 case is the rich subject of integers represented by binary quadratic forms which lies outside the scope of this article.

A natural question is then: what is the amount of surfaces that actually have a Brauer-Manin obstruction? This question was investigated in work by Mitankin \cite{mitankin2017}. To be precise, for a fixed non-zero integer $n$ consider the family $\mathcal{F}_n$ of surfaces $X_{a,b,c}: ax^2 +by^2 +cz^2 = n$ where $a,b,c \in \Z$ such that $ax^2 + by^2 +cz^2$ is indefinite and non-degenerate. The first question is how often this family has integral solutions everywhere locally. To study this one introduces the height function $H(a,b,c) = \max(| a| , | b| , | c| )$ and considers the following quantity as $B > 0$ varies:
\begin{equation*}
    N_{\text{loc}}(B) = | \{X_{a,b,c} \in \mathcal{F}_n : H(a,b,c) \leq B,  X_{a,b,c}(\mathbb{A}_{\Z}) \neq \emptyset\}|.
\end{equation*}
Then Mitankin proves \cite[Theorem 1.1]{mitankin2017} that there exist a non-zero constant $c_n$ such that
$N_{\text{loc}}(B)  \sim c_n B^3$.     
So a positive proportion of such surfaces have integral solutions everywhere locally. The second quantity they consider is
\begin{equation*}
    N_{\text{Br}}(B) = | \{X_{a,b,c} \in \mathcal{F}_n: H(a,b,c) \leq B,  X_{a,b,c}(\mathbb{A}_{\Z}) \neq \emptyset \text{ and } X_{a,b,c}(\Z) = \emptyset \}|.
\end{equation*}
For this quantity it is proved that \cite[Theorem 1.2, Corollary 1.4]{mitankin2017}:
\begin{equation*}
    B^{\frac{3}{2}} (\log B)^{-\frac{1}{2}} \ll_n N_{\text{Br}}(B) \ll_n B^{\frac{3}{2}} (\log B)^3.
\end{equation*}
So in particular $0 \%$ of these surfaces have a Brauer-Manin obstruction. In this thesis we improve both of these bounds.
\begin{theorem}
The following bounds hold
\begin{equation*}
    B^{\frac{3}{2}} (\log B)^{\frac{1}{2}} \ll_n N_{\emph{Br}}(B) \ll_n B^{\frac{3}{2}} (\log B)^\frac{3}{2}.
\end{equation*}
\label{Main result}
\end{theorem}
\subsubsection*{Acknowledgements.}
I would like to thank my thesis advisor Vladimir Mitankin for providing the topic of this thesis, answering questions related to it and the helpful comments provided while writing it. I am also thankful of Valentin Blomer for his feedback on the thesis this paper is based on. This paper is a result of the author's Master's thesis completed in the context of the Master of mathematics at the university of Bonn.
\subsubsection*{Structure}
This paper is organized as follows, in the first section we give an introduction to the integral Brauer-Manin obstruction and in particular describe the case of integral affine quadrics. In the second section we give a proof of Theorem \ref{Main result}, first of the lower and then the upper bound.

\section{Brauer-Manin obstruction}
In this section we will describe the Brauer-Manin obstruction to the Hasse principle. To do this we will define a certain subset of $X(\mathbb{A}_K)$, the Brauer-Manin set $X(\mathbb{A}_K)^{\text{Br}}$, which contains $X(K)$. We first recall the definition of the Brauer group of a scheme: $\text{Br}(X) = H^2_{\text{\'{e}t}}(X, \mathbb{G}_m)$. The cohomology on the right-hand side is \'{e}tale cohomology.

\subsection{The general Brauer-Manin obstruction.}
For this section let $K$ be a number field, $\Omega_K$ the set of places of $K$ and for $v \in \Omega_K$ we let $K_v$ be the completion of $K$ with respect to the place $v$. We recall the following facts from class field theory. For every place $v \in \Omega_K$ there exists an injective invariant map $\text{inv}_v:\text{Br}(K_v) \to \Q/\Z$ which is an isomorphism if $v$ is finite, has image $\{0, \frac{1}{2}\}$ if $v$ is real and image $\{0\}$ if $v$ is complex \cite[Theorem~7.1.4]{neukirch2013cohomology}. These invariant maps also allow us to compute the Brauer group of $K$ via the following exact sequence \cite[Theorem~8.1.17]{neukirch2013cohomology} which is known as the Brauer-Hasse-Noether theorem.

$$0 \to \text{Br}(K) \to \bigoplus_{v \in \Omega_K}\text{Br}(K_v) \xrightarrow{\sum_v \text{inv}_v} \Q/\Z \to 0.$$

Now let $X$ be a $K$-variety, $\alpha \in \text{Br}(X)$ and $(x_v)_v \in X(\mathbb{A}_K)$ an adelic point. We would like to have a condition which detects whether $(x_v)_v$ comes from a rational point. For this we note that if $(x_v)_v$ came from a rational point $x \in X(\Q)$, then the tuple of Brauer elements $(\alpha(x_v))_v$ would have to come from the Brauer element $\alpha(x)$ via the inclusions $K \to K_v$. By the Brauer-Hasse-Noether exact sequence such a thing can only happen if $\sum_{v \in \Omega_v} \text{inv}_v(\alpha(x_v)) = 0$. Doing the same for every $\alpha \in \text{Br}(X)$ shows that the set of rational points on $X$ is contained in the left kernel of the bilinear pairing 
\begin{equation}
    \label{Brauer-Manin pairing}
    \begin{split}
    X(\mathbb{A}_K) \times \text{Br}(X) &\to \Q/\Z \\ 
    ((x_v)_v, \alpha) &\mapsto \sum_{v \in \Omega_v} \text{inv}_v(\alpha(x_v)).
    \end{split}
\end{equation} This bilinear pairing might be undefined if $\alpha(x_v) \neq 0$ for more than a finite amount of places of $K$, \cite[Proposition 8.2.1]{poonen2017rational} guarantees that this does not happen. We call the left kernel of the bilinear pairing $(\ref{Brauer-Manin pairing})$ the Brauer-Manin set and denote it by $X(\mathbb{A}_K)^{\text{Br}}$.
By the previous discussion $X(K)$ injects into $X(\mathbb{A}_K)^{\text{Br}}$ so in particular if $X(\mathbb{A}_K) \neq \emptyset$ but $X(\mathbb{A}_K)^{\text{Br}} = \emptyset$, then this argument shows that $X$ has no rational points and thus fails the Hasse principle. We say that $X$ has a Brauer-Manin obstruction to the Hasse principle.

Let now $\mathcal{X}$ be an integral model of $X$, it is separated so we can identify $\mathcal{X}(\mathcal{O}_{K_v})$ with a subset of $\mathcal{X}(K_v) = X(K_v)$ and thus form a similar obstruction for integral points by considering the bilinear pairing
\begin{equation}
    \mathcal{X}(\mathbb{A}_{\mathcal{O}_K}) \times \text{Br}(X) \to \Q/\Z: ((x_v)_v, \alpha) \to \sum_{v \in \Omega_v} \text{inv}_v(\alpha(x_v)).
    \label{Integral Brauer-Manin pairing}
\end{equation} 
The integral Brauer-Manin set is the left kernel of the bilinear pairing (\ref{Integral Brauer-Manin pairing})  and we denote it by $\mathcal{X}(\mathbb{A}_{\mathcal{O}_K})^{\text{Br}}$. It is clear from the definitions that $\mathcal{X}(\mathbb{A}_{\mathcal{O}_K}) ^{\text{Br}} =  \mathcal{X}(\mathbb{A}_{\mathcal{O}_K})  \cap X(\mathbb{A}_K)^{\text{Br}}$. If $\mathcal{X}(\mathbb{A}_{\mathcal{O}_K})  \neq \emptyset$ but $\mathcal{X}(\mathbb{A}_{\mathcal{O}_K}) ^{\text{Br}} = \emptyset$ we say that $\mathcal{X}$ has a Brauer-Manin obstruction to the integral Hasse principle. The constant Brauer elements, i.e the elements which lie in the image of the pullback $\text{Br}(K) \to \text{Br}(X)$ induced by the structure morphism $X \to \text{Spec}(K)$, are contained in the right kernel of these bilinear pairings. Indeed, if we identify $\alpha \in \text{Br}(K)$ with its image in $\text{Br}(X)$, then for all $x_v \in X(K_v)$ we have $\text{inv}_v(\alpha(x_v)) = \text{inv}_v(\alpha)$. To compute the Brauer-Manin set it thus suffices to only work with representatives of $\text{Br}(X)/\text{Br}(K)$. Another useful fact for computations is that this bilinear pairing is locally constant (and thus continuous) with respect to the topology on $X(\mathbb{A}_{K})$ \cite[Corollary 8.2.11]{poonen2017rational}.
\subsection{The Brauer-Manin obstruction for affine integral quadric surfaces}
\label{procedure integral affine quadric}
We now come to the case of interest for this thesis. Let $q$ be a non-degenerate indefinite integral quadratic form of rank $3$ and $n$ a non-zero integer. The question we want to answer is whether there exist integers $x,y,z \in \Z$ such that $q(x,y,z) = n$. Let $X$ be the affine surface defined over $\Q$ by this equation and let $\mathcal{X}$ be the obvious integral model of it. It is then shown by Colliot-Thélène and Xu in \cite[Theorem 6.3]{colliot2009brauer} that the Brauer-Manin obstruction is the only obstruction to the integral Hasse principle, i.e. if $\mathcal{X}(\mathbb{A}_\Z)^{\text{Br}} \neq \emptyset$ then $\mathcal{X}(\Z) \neq 0.$ It is also shown that if $X(\Q) \neq \emptyset$ and $d= -\text{disc}(q)n$ is not a square, then $\text{Br}(X)/\text{Br}(\Q) \cong \Z/ 2 \Z$ and if $d$ is a square then $\text{Br}(X)/\text{Br}(\Q)= 0$. An explicit algorithm to find a generator of this group is given in section 5.8 of \cite{colliot2009brauer}. We now give a description of this algorithm.

Consider first the projectivization $\overline{X}$ of $X$, concretely $\overline{X} \subseteq \mathbb{P}^3$ is the projective surface defined by $q(x,y,z) = nt^2.$ Then choose a $\Q$-point
$M$ on $\overline{X}$. Let $l_1 = l_1(x,y,z,t)$ be the linear form defining the tangent plane of $\overline{X}$ at $M$. Then by \cite[\S IV: Proposition 3']{serre2012course} there exist linearly independent linear forms $l_2,l_3,l_4$ and $c \in \Q^*$ such that $$q(x,y,z) - nt^2 = l_1 l_2 + c(l_3^2 - dl_4^2).$$ Conversely, if we have such linear forms then $l_1$ defines a plane tangent to $\overline{X}$. Consider the quaternion algebra $\alpha = (\frac{l_1}{t},d) \in \text{Br}(\Q(\overline{X})) =\text{Br}(\Q(X)).$ By the identity in the $l_i$ we get in $\text{Br}(\Q(X))$ that 
\begin{equation*}
    \alpha = (\frac{l_2}{t}, d) (\frac{l_1 l_2}{t^2},d) =  (\frac{l_2}{t},d) (c,d) (\frac{l_3^2 - dl_4^2}{t^2},d) =  (\frac{l_2}{t},d) (c,d).
\end{equation*}
So $\alpha$ is defined on $U_1 = \{t l_1 \neq 0\}$ and on $U_2 = \{t l_2 \neq 0\}$. By Grothendieck purity \cite{grothendieck1968groupe} this implies that $\alpha$ comes from the Brauer group of $\{t~\neq~0\} = X$ since $\{l_1 = 0 = l_2\}$ is at least of codimension 2.
\section{The proof of Theorem \ref{Main result}}
We recall that the statement of theorem \ref{Main result} is the following inequality:
\begin{equation*}
    B^{\frac{3}{2}} (\log B)^{\frac{1}{2}} \ll_n N_{\emph{Br}}(B) \ll_n B^{\frac{3}{2}} (\log B)^\frac{3}{2}.
\end{equation*}
\subsection{The lower bound} It suffices to do the case $n=1$ since the surfaces $ax^2 + by^2 +cz^2 = 1$ and $anx^2 + bny^2 +cnz^2 = n$ are the same. We first construct a family with a Brauer-Manin obstruction. For this fix a prime $q \equiv 1 \bmod{8}$, e.g. $q = 17$. For $a,c,d,e \in \Z \setminus \{0\}$ with $c,d,e$ pairwise coprime consider the family of integral surfaces $Y_{a;c,d,e}: aq^2 c^2x^2 -a d^2y^2 + e^2 q z^2 = 1$. We start with a lemma showing when this family has a local solution everywhere, clearly we always have $Y_{a;c,d,e}(\R) \neq \emptyset$. From now on $p$ is always a prime number.

\begin{lemma}
$Y_{a;c,d,e}(\Z_p) \neq \emptyset$ for all primes $p$ if and only if $(a,eq) = 1$, $(d,q) = 1$ and for all odd primes $p\mid  a$ we have $(\frac{q}{p}) = 1$.
\end{lemma}
\begin{proof}
The only if part is clear by looking modulo $p$, the other part we do by case analysis.
\begin{itemize}
    \item If $p \nmid 2aq$, then the associated equation has a smooth point modulo $p$ which lifts to a $\Z_p$-point by Hensel's lemma.
    \item If $p\mid a$ and  if $p$ is odd, then by assumption $q$ is a non-zero square modulo $p$ which implies we can write $q = \alpha^{-2}$ for $\alpha \in \Z_p$. Similarly, if $p = 2$ and $2 \nmid  e$, then since $q \equiv 1 \pmod{8}$ we can also write $q = \alpha^{-2}$ for $\alpha \in \Z_p^{*}$. In both cases $e$ must be invertible in $\Z_p$ since $(a,e) = 1$. So $(0,0,\alpha e^{-1}) \in Y_{a;c,d,e}(\Z_p)$.
    \item If $p=q$, then we can factorize $a = \prod_{i} p_i^{\alpha_i}$ where the $p_i$ are prime numbers. Then $(\frac{a}{q}) = \prod_i (\frac{p_i}{q})^{\alpha_i} = 1$. Indeed, by our assumption and quadratic reciprocity $(\frac{p_i}{q}) = (\frac{q}{p_i}) = 1$ if $p_i$ is odd. If $p_i = 2$ then also $(\frac{2}{q}) = 1$ since $q \equiv 1 \pmod{8}.$ So $a$ is the square of a unit in $\Z_q^{*}$ which gives a point in $Y_{a;c,d,e}$ by setting $x=z=0$.
    \item If $p=2$ but $2\mid e$, then by assumption $2 \nmid  a,c,d$. A direct computation shows that the quadratic form $x'^2 -y'^2$ takes all values in $ (\Z / 8\Z)^*$, letting it take the value $a^{-1}$ and setting $z=0, x' = qcx, y' = dy$ gives a solution to $aq^2c^2 x^2 -ad^2y^2 +e^2qz^2=1$ modulo $8$ which lifts to $\Z_2$ by Hensel's lemma.
\end{itemize}
\end{proof}

From now on let $a,c,d,e$ be as in the above lemma. We want to know when $Y_{a;c,d,e}$ has no integral point, by \cite[Theorem 6.3]{colliot2009brauer} this only happens if there is a Brauer-Manin obstruction. In this case $\text{Br}(Y_{a;c,d,e}) / \text{Br}(\Q) \cong \Z/ 2 \Z$ since $-(aq^2c^2)(-ad^2)(e^2q) = q \in \Q^* / \Q^{*2}$ is not a rational square as in section \ref{procedure integral affine quadric}. We use the procedure in that section to find a representative of the generator of this group. On the projectivization of $Y_{a;c,d,e}$ given by the equation $aq^2c^2x^2 -ad^2y^2 + qe^2 z^2 =t^2$ we have a rational point $(d, -qc,0,0)$. The tangent plane to this point is defined by the equation $2aq^2c^2dx + 2ad^2qcy = 0$ so a representative for the generator of the group is given by $(qcx + dy,q)$.

We now need to find what the values of this Brauer element are when evaluated at points of $Y_{a;c,d,e}(\Z_p)$. Since the evaluation function is continuous we can restrict our attention to the dense open subset $U=\{qcx+dy \neq 0  \}$. That is, we will only consider points in $Y_{a;c,d,e}(\Z_p) \cap U(\Q_p)$. We split this computation into the following cases:
\begin{itemize}
    \item If $p=\infty$ or $2$, then $(qcx+dy,q)_{p} = 1$ since $q$ is a square in $\R$ and $\Z_{2}$.
    \item If $p \nmid  2q$ then the following equality follows from the computation of the Hilbert symbol \cite[\S III: Theorem 1]{serre2012course}.
    $$(qcx+dy,q)_p = (\frac{q}{p})^{v_p(qcx+dy)}.$$
    We may assume that $(\frac{q}{p}) = -1$ since the other case is trivial, thus $(a,p) = 1.$ If $p \mid qcx + dy$, then $p \mid  q^2c^2x^2 - d^2y^2$ so by looking at the defining equation modulo $p$ we see that $e^2qz^2 \equiv 1 \pmod{p}$ which contradicts that $(\frac{q}{p}) = -1$.
    \item If $p=q$, then by looking at the defining equation modulo $q$ we get $ad^2y^2 \equiv 1 \pmod{q}$ which shows that $y$ is non-zero modulo $q$. Then $(qcx+dy,q)_q = (\frac{dy}{q})$ because of the computation of the Hilbert symbol. So it is equal to $1$ if and only if $dy$ is a square modulo $q$. Again looking at the defining equation modulo $q$ we see that this happens if and only if $a$ is a non-zero fourth power modulo $q$.
    
\end{itemize}

Combining all this we get that $Y_{a;c,d,e}(\Z) = \emptyset$ if and only if $a$ is not a fourth power modulo $q$. On the other hand we already assumed that $a$ was a square modulo $q$ so the residue of $a$ modulo $q$ has to be an element of 
\begin{equation}
    S = (\Z / q \Z)^{*2} \setminus (\Z / q \Z)^{*4}
    \label{Definition S}.
\end{equation}
The group isomorphism $(\Z / q\Z)^* \cong \Z / (q-1) \Z$ and the fact that \\$q \equiv 1 \pmod{8}$ imply that the set $S$ is non-empty and has size $\frac{q-1}{4}$.

We now want to count the amount of such equations with coefficients smaller than $B$, denote this $N_{\text{Br}}'(B)$, i.e. the quantity $$  |\{Y_{a;c,d,e} :  H(aq^2c^2,ad^2,qe^2) \leq B,  Y_{a;c,d,e}(\mathbb{A}_{\Z}) \neq \emptyset \text{ and } Y_{a;c,d,e}(\Z) = \emptyset \}| $$
We will first introduce some notation to encode the conditions above.
First of all we will denote the indicator function of $S$ by $1_{a \in S}$, we can encode this in the standard way as a sum of characters
\begin{equation}
    1_{a\in S} =  \frac{1}{q-1} \sum_{s \in S} \sum_{\chi \text{ mod } q} \chi(s) \chi(a).
    \label{S as sum of characters.}
\end{equation}
Note that we used that $s \in S \Longleftrightarrow s^{-1} \in S$ to simplify the notation.
We will also use $\alpha(a)$ to denote the indicator function of the set $\{a \in \Z: p \mid a \Rightarrow (\frac{p}{q}) = 1\}$.
Similarly we use the notation $\beta(c,d,e)$ for the indicator function of the condition $c,d,e$ pairwise coprime.
 
Since the sign of $a$ is immaterial we have
\begin{equation}
    N_{\text{Br}}'(B) =  2\sum_{a \leq \frac{B}{q^2}} \alpha(a) 1_{a \in S} \sum_{\substack{c \leq \frac{\sqrt{B}}{q\sqrt{a}}}}  \sum_{\substack{d \leq \sqrt{\frac{B}{a}} \\ (d,q) = 1}} \sum_{\substack{e \leq \sqrt{\frac{B}{q}} \\ (e,a) = 1}} \beta(c,d,e).
    \label{NqBr definition}
\end{equation}
To remove the condition $(a,e) = 1$ we can use that its indicator function is given by $\sum_{f\mid (a,e)} \mu(f).$ We can then swap the order of the summation so (\ref{NqBr definition}) turns into
\begin{equation}
    N_{\text{Br}}'(B) = 2 \sum_{f \leq \sqrt{\frac{B}{q}} } \sum_{a \leq \frac{B}{q^2 f}}  \alpha(af) 1_{af \in S} \sum_{\substack{c \leq \frac{\sqrt{B}}{q\sqrt{af}} }}  \sum_{\substack{d \leq \sqrt{\frac{B}{af}} \\ (d,q) = 1}} \sum_{\substack{e \leq \frac{\sqrt{B}}{\sqrt{q}f}}} \beta(c,d,fe).
    \label{Lower bound eq1}
\end{equation}
We will denote the triple sum over $c,d,e$ by $V(B,a,f)$. Note that by definition $\beta(c,d,fe) = \beta(c,d,e)$ if $(f,cd) = 1$ and equal to $0$ otherwise. So we can also write this quantity as 
\begin{equation}
    V(B,a,f) = \sum_{\substack{c \leq \frac{\sqrt{B}}{q\sqrt{af}} \\ (c,f) = 1}}  \sum_{\substack{d \leq \sqrt{\frac{B}{af}} \\ (d,qf) = 1}} \sum_{\substack{e \leq \frac{\sqrt{B}}{\sqrt{q}f}}} \beta(c,d, e).
\end{equation}
It also follows from the definition of $\alpha$ that $\alpha(af) = \alpha(a) \alpha(f)$. Using this and (\ref{S as sum of characters.}) we obtain
\begin{equation*}
     N_{\text{Br}}'(B) =  \frac{2}{q-1} \sum_{s \in S} \sum_{\chi \text{ mod } q} \chi(s)  \sum_{f \leq \sqrt{\frac{B}{q}} } \alpha(f) \chi(f) \sum_{a \leq \frac{B}{q^2 f}}  \alpha(a) \chi(a) V( B,a, f).
\end{equation*}

We now introduce the following intermediary quantities for which we will find asympotic formulas one by one. First of all the sum over $a$ is denoted by 
\begin{equation}
        W_\chi(B,f) = \sum_{a \leq \frac{B}{q^2 f} } \alpha(a)  \chi(a) V(B,a,f).
        \label{Definition of W_chi}
\end{equation}
We will also consider the sum over $f$:
\begin{equation}
    U_\chi(B) = \sum_{f \leq \sqrt{\frac{B}{q}} } \alpha(f) \chi(f) W_\chi(B,f).
    \label{Definition of U_chi}
\end{equation}
With this notation the quantity which we have to compute is 
\begin{equation}
    N_{\text{Br}}'(B) = \frac{2}{q-1} \sum_{s \in S} \sum_{\chi \text{ mod } q} \chi(s) U_\chi(B).
    \label{N_Br as a sum over U_chi}
\end{equation}
We begin by evaluating $V(B,a,f)$. For this we need the following lemma. A similar result is the subject of \cite{toth2002probability}, in particular if $X=Y=Z$ and $a=b=c$ the following lemma is a special case of the main theorem of \cite{toth2002probability}.
\begin{lemma}
Let $a,b,c$ be non-zero integers and let $X,Y,Z \geq 1$. Then there exists a non-zero constant $C_{a,b,c} = \prod_p C_p$ only depending on $a,b,c$ such that 
\begin{equation*}
\begin{split}
    L_{a,b,c}(X,Y,Z) &= \sum_{\substack{x \leq X \\ (x,a) = 1}}  \sum_{\substack{y \leq Y \\ (y,b) = 1}} \sum_{\substack{z \leq Z \\ (z,c) = 1}} \beta(x,y,z) = C_{a,b,c} XYZ \\ &+ O\left((\frac{\tau(a) (\log X)^2}{X} + \frac{\tau(b) (\log Y)^2}{Y} +\frac{\tau(c) (\log Z)^2}{Z}) XYZ\right).
    \end{split}
\end{equation*}
For any prime $p$ the constant $C_p$ depends on how many of the $a,b,c$ are divisible by the prime $p$. Namely , if $p$ divides exactly $i$ of them, then $C_p = (1 - \frac{1}{p})^{2}(1 + \frac{2-i}{p})$.
\label{Chance of pairwise coprime}
\end{lemma}
In the proof of this lemma we will require the notion of a multivariable multiplicative function. A survey of this notion is the subject of \cite{toth2014multiplicative}.
\begin{definition}
A function $f:\Z_{> 0}^r \to \C$ is multiplicative if it is not identically zero and
\begin{equation*}
    f(m_1 n_1, \cdots, m_r n_r) = f(m_1, \cdots, m_r) f(n_1, \cdots, n_r)
\end{equation*}
holds for any $m_1, \cdots m_r, n_1, \cdots n_r \in \Z_{> 0}$ such that $(m_1 \cdots m_r, n_1 \cdots n_r) = 1.$
\end{definition}
If $f$ is multiplicative then $f(1, \cdots, 1) = 1$ and
\begin{equation*}
    f(n_1, \cdots, n_r) = \prod_p f(p^{v_p(n_1)}, \cdots, p^{v_p(n_r)}).
\end{equation*}
A multiplicative function $f$ is thus completely decided by its values at       \\ $f(p^{\alpha_1}, \cdots, p^{\alpha_r})$ for $p$ a prime and $\alpha_1, \cdots \alpha_r \in \Z_{\geq 0}$.
We now prove Lemma \ref{Chance of pairwise coprime}.
\begin{proof}
We will first encode the conditions $x,y,z$ pairwise coprime in a more useful way. Note that $\beta(x,y,z)$ is multiplicative. We can encode this condition as
\begin{equation}
   \beta(x,y,z) =  \sum_{\substack{t \mid (x,y,z)}}\sum_{\substack{u \mid  (x,y)}}  \sum_{\substack{v \mid  (x,z)}} \sum_{\substack{w \mid (y,z)}} \mu(uvwt)\mu(t) \tau(t).
   \label{Encoding of pairwise coprime}
\end{equation}
Indeed the right hand side is also multiplicative for the same reason that the Dirichlet convolution of two single variable multiplicative functions is multiplicative. It thus suffices to check the equality in the case that $x,y,z$ are prime powers and this is a simple computation.

Use this encoding (\ref{Encoding of pairwise coprime}) of $\beta$ in the definition of $L_{a,b,c}(X,Y,Z)$. Since the terms are non-zero unless $t,u,v,w$ are pairwise coprime we can switch the order of summation and get that
\begin{equation}
    L_{a,b,c}(X,Y,Z) = \sum_{\substack{uvt \leq X, \  (uvt,a) = 1 \\ uwt \leq Y, \ (uwt,b) = 1 \\ vwt \leq Z, \ (vwt,c)=1}} \mu(uvwt)\mu(t) \tau(t) \sum_{\substack{x \leq \frac{X}{uvt} \\ (x,a) = 1}}  \sum_{\substack{y \leq \frac{Y}{uwt} \\ (y,b) = 1}} \sum_{\substack{z \leq \frac{Z}{vwt} \\ (z,c) = 1}} 1.
    \label{pairwise coprime eq 1}
\end{equation}
Now for the inner sums use the standard fact that $$\sum_ {\substack{x \leq X \\ (x,a) = 1}} 1= \sum_{x \leq X} \sum_{d |(x,a)} \mu(d) = \sum_{d |a}( \frac{ \mu(d)X}{d} + O(1)) = \frac{\phi(a)X}{a} + O(\tau(a)).$$ Applying this and the trivial inequality $\sum_ {\substack{x \leq X \\ (x,a) = 1}} 1 \ll X$ if $X \leq \tau(a)$ thrice one finds that 
\begin{equation*}
\begin{split}
    \sum_{\substack{x \leq \frac{X}{uvt} \\ (x,a) = 1}}  \sum_{\substack{y \leq \frac{Y}{uwt} \\ (y,b) = 1}} \sum_{\substack{z \leq \frac{Z}{vwt} \\ (z,c) = 1}} 1 &= \frac{XYZ}{u^2 v^2 w^2 t^3} \Big (\frac{\phi(a)\phi(b)\phi(c)}{abc} \\ &+ O(\frac{\tau(a)uvt}{X} + \frac{\tau(b)uwt}{Y} +\frac{\tau(c)vwt}{Z}) \Big ).
    \end{split}
\end{equation*}
If we apply this to \eqref{pairwise coprime eq 1} and uses trivial inequalities for the sums over the error terms we get that
\begin{equation*}
\begin{split}
    L_{a,b,c}(X,Y,Z) =& K_{a,b,c}(X,Y,Z) XYZ \Big (\frac{\phi(a)\phi(b)\phi(c)}{abc}  \\ +&O(\frac{\tau(a) (\log X)^2}{X} + \frac{\tau(b) (\log Y)^2}{Y} +\frac{\tau(c) (\log Z)^2}{Z})\Big ).
\end{split}
\end{equation*}
Where 
\begin{equation}
    K_{a,b,c}(X,Y,Z) = \sum_{\substack{uvt \leq X, \  (uvt,a) = 1 \\ uwt \leq Y, \ (uwt,b) = 1 \\ vwt \leq Z, \ (vwt,c)=1}} \frac{\mu(uvwt)\mu(t) \tau(t)}{u^2v^2w^2t^3}.
    \label{definition K}
 \end{equation}

We now show that $K_{a,b,c}(X,Y,Z)$ converges as $X, Y, Z \to \infty$ and decide its speed of convergence. For this we have to bound the size of 
\begin{equation*}
    \left| \sum_{\substack{uvt \geq X, \  (uvt,a) = 1 \\ \text{or }uwt \geq Y, \ (uwt,b) = 1 \\ \text{or } vwt \geq Z, \ (vwt,c)=1}} \frac{\mu(uvwt)\mu(t) \tau(t)}{u^2v^2w^2t^3}  \right|
    \leq \sum_{t} \frac{\tau(t)}{t^3} \sum_{\substack{uv \geq \frac{X}{t} \\ \text{or }uw \geq \frac{Y}{t} \\ \text{or } vw \geq \frac{Z}{t}}} \frac{1}{u^2v^2w^2}.
  \label{relative prime constant speed of convergence}
\end{equation*}
We can bound this inner sum as 3 sums over $uv \geq \frac{X}{t}$, etc. Since they are analogous we only do one of these and by putting $n = uv$ we can bound this term as
\begin{equation*}
    \sum_{t} \frac{\tau(t)}{t^3} \sum_{w} \frac{1}{w^2} \sum_{n \geq \frac{X}{t}}\frac{\tau(n)}{n^2} \ll  \sum_{t} \frac{\tau(t)}{t^3} \sum_{w} \frac{1}{w^2} \frac{t \log(\frac{X}{t})}{X} \ll  \frac{\log X}{X}.
\end{equation*}
Where we have used partial summation and the classical fact \cite[\S I.3.2]{tenenbaum46introduction} that $$\sum_{n \leq X} \tau(n) = X\log X + O(X)$$ two times. This gives a total error term of size $O(\frac{\log X}{X} +\frac{\log Y}{Y} +\frac{\log Z}{Z}).$

It thus only remains to compute the value of the completed sum
\begin{equation}
    \sum_{\substack{u,v,w,t \\ (uvt,a) = (uwt,b) = (vwt,c) = 1}} \frac{\mu(uvwt)\mu(t) \tau(t)}{u^2v^2w^2t^3}.
\end{equation}
Because the terms are multiplicative in $u,v,w,t$, \cite[Proposition 11]{toth2014multiplicative} implies that this sum convergences to the Euler product $\prod_{p} C'_p$. The value of $C'_p$ only depends on how many of the $a,b,c$ the prime $p$ divides. Let $i$ be this amount.
\begin{itemize}
    \item If $i=0$ then $C'_p = 1 - \frac{3}{p^2} + \frac{2}{p^3} = (1- \frac{1}{p})^2(1+ \frac{2}{p})$.
    \item If $i = 1$ then $C'_p = 1- \frac{1}{p^2} = (1- \frac{1}{p})(1+ \frac{1}{p})$
    \item If $i = 2,3$ then $C'_p = 1$.
\end{itemize}
It remains to prove that $C_{a,b,c} = \frac{\phi(a)\phi(b)\phi(c)}{abc} \prod_p C'_p$. This is clearly true since $C_p = C'_p(1-\frac{1}{p})^i$. This completes the proof.
\end{proof}
From this lemma it now follows by taking $X = \frac{ \sqrt{B}}{q \sqrt{af}}, Y = \frac{\sqrt{B}}{\sqrt{af}}, Z = \frac{\sqrt{B}}{\sqrt{q} f}$ and $a = f, b= qf, c =1$ that 
\begin{equation}
  V(B,a,f) = \frac{B^{\frac{3}{2}}}{(q^{\frac{3}{2}}a f^2}\left(C_f + O(\frac{\tau(f) }{\sqrt{B}}(f(\log \frac{B}{f^2})^2 + \sqrt{af}(\log \frac{B}{af})^2\right),
    \label{Triple sum for lower bound}
\end{equation}
where 
\begin{equation*}
    C_f = \frac{q+1}{q+2} \prod_{p \mid f}(1 + \frac{2}{p})^{-1} \prod_{ p}(1-\frac{3}{p^2} + \frac{2}{p^3}).
\end{equation*}

We will now compute $W_\chi(B,f)$. Using the formula (\ref{Triple sum for lower bound}) in the definition of $W_\chi(B,f)$ (\ref{Definition of W_chi}) we find that \begin{equation}
\begin{split}
    W_\chi(B,f) &= \left(\frac{B^{\frac{3}{2}}C_f}{q^{\frac{3}{2}}f^2} + O(\frac{\tau(f) B (\log B)^2 }{f})\right)\sum_{a \leq \frac{B}{q^2 f}} \frac{\chi(a) \alpha(a)}{a} \\ &+ O(\frac{B }{f^ {\frac{3}{2}}}\sum_{a \leq \frac{B}{q^2 f}} \frac{\alpha(a)}{a^{\frac{1}{2}}} (\log \frac{B}{a})^2).
        \end{split}
    \label{Formula for W_chi}
\end{equation}

To compute this we will first compute $\sum_{a \leq \frac{B}{q^2 f}} \chi(a) \alpha(a)$ and then apply partial summation. This sum will be evaluated using the Selberg-Delange method. See \cite[II.5 Theorem~3]{tenenbaum46introduction} for a precise statement of the Selberg Delange method and the preceding section for the definition of type $\mathcal{T}$.

\begin{lemma}
Let $\chi$ be a character modulo $q$ and let $\psi$ be the character $(\frac{\_}{q})$. Then we have the following. If $\chi = \psi$ or if $\chi$ is principal, then
\begin{equation*}
\sum_{\substack{a \leq x}} \chi(a)\alpha(a) = D x(\log x)^{-\frac{1}{2}} + O(x (\log x)^{-\frac{3}{2}}),
\end{equation*}
where
$D = \pi^{-\frac{1}{2}} (1- \frac{1}{q})^{\frac{1}{2}} \prod_{p}(1-p)^{-\frac{\psi(p)}{2}}$. Otherwise
\begin{equation*}
    \sum_{a \leq x }  \chi(a) \alpha(a) = O(x e^{{-d_1 \sqrt{\log x}}})
\end{equation*}
for some constant $d_1 > 0$.
\label{Selberg-Delange}
\end{lemma}
\begin{proof}
We first do the case $\chi= \psi$ or $\chi$ is principal. If $\chi = \psi$, then the only terms in the sum for which $\alpha(a)$ is non-zero, are by definition the ones such that for all primes $p|a$ we have $\psi(p) = 1$. Since $\psi$ is multiplicative this implies that $\psi(a) = 1$. Now look at the associated Dirichlet series of this sum
\begin{equation*}
\begin{split}
    F(s) = \sum_{n}  \alpha(n)n^{-s} & = \prod_{\substack{p \\ \psi(p)=1}}(1-p^{-s})^{-1}  \\ &= (1-\frac{1}{q^{s}})^{\frac{1}{2}}\zeta(s)^{\frac{1}{2}}L(\psi,s)^{\frac{1}{2}}\prod_{\substack{p \\ \psi(p)=-1}}(1-p^{-2s})^{\frac{1}{2}}.
\end{split}
\end{equation*}
We were able to write the above as this Euler product since $\alpha$ is totally multiplicative. 
In the following we write $K(s) =\prod_{\substack{p,  \psi(p)=-1}}(1-p^{-2s})^{\frac{1}{2}}$. We will use the classical notation $s = \sigma + it$. For $\sigma > \frac{1}{2}$ we have the inequality 
\begin{equation*}
    \prod_{\substack{p,  \psi(p)=-1}}|(1-p^{-2s})^{\frac{1}{2}}| \leq \prod_p (1 + p^{-2 \sigma})^{\frac{1}{2}}  \leq \prod_p (1 - p^{-2 \sigma)})^{-\frac{1}{2}} = \zeta( 2 \sigma)^{\frac{1}{2}}.
\end{equation*}
  So $K(s)$ is holomorphic and bounded by $\zeta(\frac{3}{2})^{\frac{1}{2}}$ in the region $\text{Re}(s) > \frac{3}{4}$.
 It is known \cite[Chapter 14]{davenportmultiplicative} that there exist some constant $c_0$ such that for every character $\chi$ modulo $q$ the Dirichlet series $L(\chi,s)$ has no zeroes in the region
  $$\sigma \geq 1 - \frac{c_0}{1 + \log(1 + |t|)},$$
  One can get rid of the possible Siegel zero by taking a smaller $c_0$. We will also require the bound $L(\sigma + i t,\psi) \leq (\frac{q|\sigma + it|}{2\pi})^{\frac{3- 2\sigma}{4}} \zeta(\frac{3}{2}) \leq (\frac{q}{2\pi})(1 + |t|)^{\frac{1}{2}} \zeta(\frac{3}{2})$ for $\frac{1}{2} \leq \sigma \leq 1$, which is the case $\eta = \frac{1}{2}$ of \cite[Theorem 3]{rademacher1958phragmen}. Here we used the fact that $q > 2\pi$. Let $c_1 = \min(c_0, \frac{1}{4})$, then the preceding discussion implies that $F(s)$ is of type $\mathcal{T}(\frac{1}{2}, \frac{1}{2}, c_1, \frac{3}{4}, (1+\frac{1}{q^{\frac{3}{4}}})^{\frac{1}{2}} (\frac{q}{2\pi} )^{\frac{1}{2}} \zeta(\frac{3}{2})).$ The Selberg-Delange method immediately implies the desired asymptotic formula.
  For the other characters $\chi$ we use the Selberg-Delange method once again. The associated Dirichlet series is
  \begin{equation*}
   \begin{split}
       F(s) = \sum_{n} \chi(n) \alpha(n)n^{-s} & = \prod_{\substack{p \\ \psi(p)=1}}(1- \chi(p)p^{-s})^{-1} \\ & = L(\psi,s )^{\frac{1}{2}}L(\chi \psi,s)^{\frac{1}{2}}\prod_{\substack{p \\ \psi(p)=-1}}(1-p^{-2s})^{\frac{1}{2}}.
   \end{split} 
   \end{equation*}
   So for similar reasons as before $F(s)$ is of type $\mathcal{T}(0, \frac{1}{2}, c_1, \frac{1}{2}, \frac{q}{2\pi} \zeta(\frac{3}{2})^{\frac{3}{2}})$ which by the Selberg-Delange method implies the desired bound.
  \end{proof}

   If we first apply this lemma for the principal character to the sum over $a$ in the error terms of (\ref{Formula for W_chi}) and apply partial summation we get that this contributes an error of size $$O(\frac{B \tau(f)}{f}(\log B)^3  + \frac{B^{\frac{3}{2}}\tau(f)}{f^2} (\log\frac{B}{f})^{-\frac{1}{2}}).$$ We then apply this lemma and partial summation to the sum over $a$ in the main term of (\ref{Formula for W_chi}). Combining this with the error above gives us that 
\begin{equation}
    W_\chi(B,f) = \delta_\chi C_f 2D\frac{B^{\frac{3}{2}}}{q^{\frac{3}{2}}f^2}(\log\frac{B}{f})^{\frac{1}{2}} + O(\frac{B \tau(f)}{f}(\log B)^3  + \frac{B^{\frac{3}{2}}\tau(f)}{f^2} (\log B)^{-\frac{1}{2}}).
    \label{Asymptotic formula for W}
\end{equation}
Where $\delta_\chi = 1$ if $\chi$ is principal or $(\frac{\_}{q})$ and zero otherwise.

The next step is to compute $U_\chi(B)$. If we use (\ref{Asymptotic formula for W}) in its definition (\ref{Definition of U_chi}) and bound the sum over the error terms using the divisor bound $\tau(f) \ll_{\epsilon} f^{\epsilon}$ \cite[Corollary I.5.1.1]{tenenbaum46introduction} with e.g. $\epsilon = \frac{1}{4}$ and the trivial bounds $|\chi(f) \alpha(f)| \leq 1$ we find that 
\begin{equation}
U_\chi(B) =  \delta_\chi \frac{2D}{q^{\frac{3}{2}}} B^{3/2} \sum_{f \leq \sqrt{\frac{B}{q}}} \frac{\mu(f)  C_f}{f^2} (\log \frac{B}{f})^{\frac{1}{2}} \alpha(f)\chi(f) + O(B^{\frac{3}{2}} (\log B)^{-\frac{1}{2}}).
   \label{Lower bound before sum f}
\end{equation}

Note that the only relevant cases are when $\chi$ is principal or $(\frac{\_}{q})$ since otherwise $\delta_\chi = 0$. In both cases the only non-zero terms are when $\alpha(f) = 1$, i.e. when for every prime $p | f$ one has $(\frac{p}{q}) = 1$. In this case $(\frac{f}{q}) = 1$ so in both cases $\chi(f) = 1$. By the definition of $C_f$ we have $C_f = \prod_{p \mid f}(1 + \frac{2}{p})^{-1} C_1$. Using trivial bounds we see that the sum $\sum_{f \leq \sqrt{\frac{B}{q}} } \frac{\mu(f)}{f^2}\prod_{p \mid f}(1 + \frac{2}{p})^{-1} \alpha(f)$ converges to its Euler product.
\begin{equation}
     \sum_{f \leq \sqrt{\frac{B}{q}} } \frac{\mu(f)}{f^2}\prod_{p \mid f}(1 + \frac{2}{p})^{-1} \alpha(f) =  \prod_{\substack{p \\ \psi(p) =1}}(1-\frac{1}{p(p+2)}) + O(B^{-\frac{1}{2}}).
     \label{Sum f without log}
\end{equation}

We then apply partial summation to the sum \eqref{Lower bound before sum f} and by \eqref{Sum f without log} we get
\begin{equation}
U_\chi(B) = \delta_\chi \frac{2C_1 D}{q^{\frac{3}{2}}}\prod_{\substack{p \\ \psi(p) =1}}(1-\frac{1}{p(p+2)}) B^{\frac{3}{2}}  (\log B)^{\frac{1}{2}} + O(B^{\frac{3}{2}} (\log B)^{-\frac{1}{2}}).
\label{Asymptotic formula for U}
\end{equation}
Let us write $E =  \frac{2C_1 D}{q^{\frac{3}{2}}}\prod_{\substack{p \\ \psi(p) =1}}(1-\frac{1}{p(p+2)})$ to simplify the notation. It then only remains to apply this formula to (\ref{N_Br as a sum over U_chi}), this gives
\begin{equation*}
    N'_{\text{Br}}(B) =  \frac{2 E}{q-1}\sum_{s \in S} (1 + (\frac{s}{q})) B^{\frac{3}{2}}  (\log B)^{\frac{1}{2}} + O(B^{\frac{3}{2}} (\log B)^{-\frac{1}{2}}).
\end{equation*}
But by the definition of $S$ \eqref{Definition S} we know that $(\frac{s}{q}) = 1$ for $s\in S$ and that $|S| = \frac{q-1}{4}$ so we can conclude that 
\begin{equation}
    N'_{\text{Br}}(B) =  E B^{\frac{3}{2}}  (\log B)^{\frac{1}{2}} + O(B^{\frac{3}{2}} (\log B)^{-\frac{1}{2}}).
\end{equation}
Since $N'_{\text{Br}}(B) \leq N_{\text{Br}}(B)$ this implies the desired bound \begin{equation*}
    B^{\frac{3}{2}} (\log B)^{\frac{1}{2}} \ll_n N_{\text{Br}}(B).
\end{equation*}
\begin{remark}
One might try to prove a better lower bound by considering multiple such families for varying primes $q$ and adding all of these together. This will give no improvement since as $q$ varies, $C_1$ and $\prod_{\substack{p \\ \psi(p) =1}}(1-\frac{1}{p(p+2)})$ are bounded and $D$ goes up to $L((\frac{\_}{q}), 1)^{\frac{1}{2}}\ll \log q$ as can be seen from the proof of Lemma \ref{Selberg-Delange}.
\end{remark}
\subsection{The upper bound.}
To find an upper bound we will use the following lemma from \cite{mitankin2017}.
\begin{lemma}
Let $a,b,c \in \Z$, if there exists an odd prime $p$ such that $v_p(a)$ is odd and $p \nmid  bcn$ then $X_{a,b,c}: ax^2 + by^2 +cz^2 = n$ has no integral Brauer-Manin obstruction.
\label{Condition upper bound}
\end{lemma}

We can thus bound $N_{\text{Br}}(B)$ by the amount of triples $(a,b,c) \in \Z^{3} \cap [-B,B]^3$ such that for all prime divisors $p\mid a$ the integer $v_p(a)$ is even or $p\mid bcn$ and such that $ax^2 + by^2 +cy^2 =n$ has local solutions everywhere. Similarly, for $b$ and $c$. Such a triple can be written as 
\begin{equation*}
    \begin{split}
        a &= v_1 u_{12} u_{13} w_{12} w_{13} w_{21}^2 w_{31}^2a_1^2, \\
        b &= v_2 u_{12} u_{23} w_{21} w_{23} w_{12}^2 w_{32}^2 b_1^2, \\
        c &= v_3 u_{13} u_{23} w_{31} w_{32} w_{13}^2 w_{23}^2 c_1^2.
    \end{split}
\end{equation*}
Here the $v_i$ have only prime factors dividing $2n$ and the $u_{ij}$ are odd squarefree. Such a decomposition can be found as follows:
\begin{itemize}
    \item The $v_i$ contain all the prime factors of $a,b,c$ dividing $2n$, they also have the same sign as $a,b,c$.
    \item The positive squarefree number $u_{12}$ is the product of the primes $p |(a,b)$, not counted with multiplicity, such that both $v_p(a)$ and $v_p(b)$ are odd. Completely analogous for $u_{13}, u_{23}.$
    \item The positive number $w_{12}$ is the product of the prime numbers $p | (a,b)$, also not counted with multiplicity, such that $v_p(a)$ is odd but $v_p(b)$ is even. The other $w_{ij}$ are analogous.
    \item The products of the prime factors that are left are squares because of the conditions coming from Lemma \ref{Condition upper bound} and can thus be written as $a_1^2, b_1^2, c_1^2.$
\end{itemize}

Because the equation $ax^2 + by^2 +cy^2 =n$ needs to have solutions locally everywhere we see that for $p\mid u_{12}$ we need $(\frac{v_3 u_{13} u_{23} w_{12} w_{13} n }{p}) =1$ by looking modulo $p$. We also have analogous conditions for $u_{13}, u_{23}$. Let now $\epsilon(v)$ be the indicator function of the set $\{v: p\mid v \Rightarrow p\mid 2n\}$ and $\delta(u ; v)$ be the indicator function of the set \begin{equation}
    \{u :u \text{ squarefree, odd and } p\mid u \Rightarrow (\frac{v}{p}) = 1\}.
    \label{Definition delta}
\end{equation} The signs of $a, b, c$ are immaterial in these conditions so we can assume that $a,b,c \geq 0$ and $u_{12}, u_{13}, u_{23} \geq 1$. Summing over $a_1, b_1, c_1$ and using that there are $O(B^{\frac{1}{2}})$ squares less than $B$ we get the following upper bound for $N_{\text{Br}}(B)$:
\begin{equation}
    \ll B^{\frac{3}{2}} \sum_{v_i \leq B}\frac{\epsilon(v_i)}{\sqrt{v_i}} \sum_{w_{ij} \leq B} \frac{1}{\sqrt{w_{ij}}^3} T(B; v_1 w_{12} w_{13} n, v_2 w_{21} w_{23} n , v_3 w_{31} w_{32} n).
    \label{The upper bound first eq}
\end{equation}
Here
\begin{equation}
    T(B;k,l,m) = \sum_{ u_{12}, u_{13}, u_{23} \leq B} \frac{\delta(u_{23}; k u_{12} u_{13}) \delta(u_{13}; l u_{12} u_{23}) \delta(u_{12}; m u_{13} u_{23})}{u_{12} u_{13} u_{23}}.
    \label{T definition}
\end{equation}
Now to bound $T$ we will first look at the related quantity
\begin{equation}
    S(X,Y,Z;k,l,m) = \sum_{ \substack{u_{23} \leq X \\ u_{13} \leq Y \\ u_{12} \leq Z}}  \delta(u_{23}; k u_{12} u_{13}) \delta(u_{13}; l u_{12} u_{23} ) \delta(u_{12}; m u_{13} u_{23}).
    \label{S definition}
\end{equation}
In particular, we will prove the following lemma.
\begin{lemma}
For $X,Y,Z \geq 2$ and $k,l,m \in \Z \setminus \{0\}$ we have the bound
\begin{equation*}
\begin{split}
      S(X,Y,Z;k,l,m)
      &\ll ( klm) ^{\frac{1}{4}}XYZ[(\log X \log Y \log Z)^{-\frac{1}{2}} \\ &+ (\log X \log Y \log Z)^{-2}((\log X)^{\frac{5}{2}} + (\log Y)^{\frac{5}{2}}  + (\log Z)^{\frac{5}{2}})].
\end{split}
\end{equation*}
\label{Lemma for bound of S}
\end{lemma}
Now assuming this lemma and applying partial summation thrice we get the inequality
\begin{equation*}
    T(B;k,l,m) \ll |klm|^{\frac{1}{4}} (\log B)^{\frac{3}{2}}.
\end{equation*}
Applying this to (\ref{The upper bound first eq}) and summing over the $w_{ij}$ we find that
\begin{equation*}
     N_{\text{Br}}(B) \ll_n B^{\frac{3}{2}} (\log B)^{\frac{3}{2}}\sum_{v_i \leq B}\frac{\epsilon(v_i)}{v_i^{\frac{1}{4}}}.
\end{equation*}
Since all the terms are positive and $\epsilon$ is completely multiplicative we can complete the sum and write it as an Euler product. We conclude that
\begin{equation*}
     N_{\text{Br}}(B) \ll_n B^{\frac{3}{2}} (\log B)^{\frac{3}{2}}\sum_{v_i}\frac{\epsilon(v_i)}{v_i^{\frac{1}{4}}} = B^{\frac{3}{2}} (\log B)^{\frac{3}{2}} \prod_{p\mid 2n}(1 - p^{-\frac{1}{4}})^{-3}
\end{equation*}
as desired.

We will now prove Lemma \ref{Lemma for bound of S}. 

\begin{proof}
Note that $S(X,Y,Z;k,l,m)$ does not change if we permute \\ $(X,k), (Y,l), (Z,m)$. Moreover, we have the trivial inequality \begin{equation}
    S(X,Y,Z;k,l,m) \leq XYZ.
    \label{S trivial inequality}
\end{equation}
If $\max(\log X, \log Y, \log Z) \geq (\min(\log X \log Y, \log X \log Z, \log Y \log Z))^4$, \\ then we can use the trivial inequality: let us assume that $X \geq Y,Z$ so by assumption $\log X \geq (\log Y \log Z)^{4}$. We then find that
\begin{equation*}
    S(X,Y,Z;k,l,m) \ll XYZ (\log X)^{\frac{1}{2}} (\log Y \log Z)^{-2}.
    \label{S when one of X,Y,Z is larger}
\end{equation*}
In this case the lemma is true.

We now assume that \begin{equation}
     \max(\log X, \log Y, \log Z) \leq \min(\log X \log Y, \log X \log Z, \log Y \log Z)^4.
     \label{assumption for proof of S}
\end{equation}
We can rewrite $S(X,Y,Z;k,l,m)$ using the equality
\begin{equation}
    \delta(u ; v) = 1_{(u,2v) = 1}\frac{\mu^2(u)}{\tau(u)}\sum_{d \mid  u} (\frac{v}{d}).
    \label{rewrite delta}
\end{equation}
By $1_{(u,2v) = 1}$ we mean the indicator function of the set $\{u: (u,2v) = 1\}$. This equality is true since both sides are multiplicative in $u$ and it is trivial when $u$ is a prime power from the definition \eqref{Definition delta}. Then we can write e.g $u_{23} =d_1 f_1$ to get that 
\begin{equation}
\begin{split}
&S(X,Y,Z;k,l,m) 
 \\ &= \sum_{ \substack{d_1 f_1 \leq X \\ d_2 f_2 \leq Y \\ d_3 f_3 \leq Z}} \frac{\mu(2 d_1 d_2 d_3 f_1 f_2 f_3)^2}{\tau(d_1 d_2 d_3 f_1 f_2 f_3)}  (\frac{k d_2 f_2 d_3 f_3}{d_1})  (\frac{l d_1 f_1 d_3 f_3 }{d_2}) (\frac{m d_1 f_1 d_2 f_2}{d_3}).
\label{S eq1}
\end{split}
\end{equation}
Now by quadratic reciprocity and since the $d_i$ are odd the factors $$(\frac{k d_2 d_3}{d_1})  (\frac{l d_1 d_3}{d_2}) (\frac{m d_1 d_2}{d_3})$$ are Dirichlet characters modulo $k,l,m$ multiplied by something which only depends on the classes of $d_1, d_2, d_3$ modulo 4. We can thus write this as a sum of characters
\begin{equation}
    (\frac{k d_2 d_3}{d_1})  (\frac{l d_1 d_3}{d_2}) (\frac{m d_1 d_2}{d_3}) = \sum_{\chi_1} \sum_{\chi_2} \sum_{\chi_3} a_{\chi_1, \chi_2, \chi_3} \chi_1(d_1) \chi_2(d_2) \chi_3(d_3).
    \label{Sum of characters}
\end{equation}
Here $\chi_1, \chi_2, \chi_3$ respectively range over all characters modulo $4l,4k,4m$ and $a_{\chi_1, \chi_2, \chi_3}$ are some complex constants.
After applying the equality \eqref{Sum of characters} to \eqref{S eq1} and switching the sums we see that we only have to bound each term corresponding to the characters $\chi_1, \chi_2,\chi_3$ separately. The corresponding term is
\begin{equation}
    \sum_{ \substack{d_1 f_1 \leq X \\ d_2 f_2 \leq Y \\ d_3 f_3 \leq Z}} \frac{\mu(2 d_1 d_2 d_3 f_1 f_2 f_3)^2}{\tau(d_1 d_2 d_3 f_1 f_2 f_3)}  \chi_1(d_1)(\frac{f_2 f_3}{d_1}) \chi_2(d_2) (\frac{ f_1 f_3 }{d_2}) \chi_3(d_3)(\frac{ f_1 f_2}{d_3}).
    \label{Term to be bounded for upper bound}
\end{equation}
A similar sum but with only 4 variables was investigated in \cite{Friedlander2010}. We will follow their approach, for this we will need Lemma~1 and 2 from that paper.
\begin{lemma}
Let $\alpha_n, \beta_{m}$ be complex numbers supported on odd integers of absolute value $\leq 1$.
We have the inequality
\begin{equation*}
    \sum_{\substack{n \leq N \\ m  \leq M }} \alpha_n \beta_{m}( \frac{m}{n}) \leq (N^{\frac{5}{6}}M + NM^{\frac{5}{6}}) (\log NM)^{\frac{7}{6}}.
    \label{S for large sums}
\end{equation*}
\end{lemma}
\begin{lemma}
Let $\chi \bmod{q}$ be a Dirichlet character and $d$ an integer such that $(d,q)=1$, then for $x \geq 2$ and for all  $C > 0$ we have
\begin{equation*}
\begin{split}
    \sum_{\substack{n \leq x \\ (n,d) = 1}} \frac{\mu(n)^2}{\tau(n)}\chi(n) &= \delta_{\chi} c(dq) \frac{x}{\sqrt{\log x}}\{ 1 + O(\frac{(\log \log 3dq)^{\frac{3}{2}}}{\log x})\} \\ &+ O_C(\tau(d) qx (\log x)^{-C}).
\end{split}
\end{equation*}
Here $\delta_\chi = 1$ if $\chi$ is principal, $\delta_\chi = 0$ otherwise and
\begin{equation*}
    c(r) = \pi^{-\frac{1}{2}}\prod_p(1 + \frac{1}{2p})(1-\frac{1}{p})^{\frac{1}{2}} \prod_{p \mid r}(1 + \frac{1}{2p})^{-1}. 
\end{equation*}
\label{S for short sums}
\end{lemma}
In particular when $d =1$ and $\chi$ is principal this gives
    \begin{equation}
     \sum_{\substack{n \leq x }} \frac{\mu(n)^2}{\tau(n)} \ll x (\log x)^{-\frac{1}{2}}.
     \label{principal case}
     \end{equation}

Let now $V \geq 1$ be a parameter which will be chosen later as a negative power of $|klm|$ times a large power of $\log (XYZ)$. We will split the sum \eqref{Term to be bounded for upper bound} into different regions which we will bound separately. The regions are as follows:
\begin{enumerate}[label = (\arabic*)]
    \item The first regions are those of the form $d_i, f_j > V$ where $i \neq j$. All of these sums are analogous so we may assume that $i=1, j=2$. Certain regions will then be counted double, to use inclusion-exclusion we will thus be required to bound regions of the type $d_1, f_2 > V$ and possible conditions on $d_2, d_3, f_1, f_3.$
    \item Another region is $d_1, d_2, d_3, f_1, f_2, f_3 \leq V$.
    \item The third type of region we consider is $d_i, d_j, f_i, f_j \leq V$ and $d_k, f_k > V$ for $\{i,j,k\}= \{1,2,3\}$.
    \item The last regions are given by $d_1, d_2, d_3 \leq V$ and $f_1, f_2, f_3 \leq V$. This will count the terms in the region in (2) twice so it will have to be subtracted.
\end{enumerate}
We consider first the regions of the form $d_1, f_2 > V$ with a possible condition on $d_2, d_3, f_1, f_3.$ Now look at this region in \eqref{S eq1}, move the sum over $d_1, f_2$ to the inside and apply trivial bounds to the terms which only depend on $d_2, d_3, f_1, f_3$, this also removes all the coprimality conditions except for $(d_1,f_2) = 1$. After doing this we find that this region is bounded by
\begin{equation*}
     \ll \sum_{f_1 \leq X V^{-1}} \sum_{d_2 \leq Y V^{-1}} \sum_{d_3 f_3 \leq Z} |\sum_{V < d_1 \leq \frac{X}{f_1}} \sum_{V < f_2 \leq \frac{Y}{d_2}} \frac{\mu(2d_1 f_2)^2}{ \tau(d_1 f_2)} \chi_1(d_1)(\frac{f_2 f_3}{d_1})|.
\end{equation*} 

Now applying Lemma \ref{S for large sums} to the sum over $d_1$ and $f_2$ we find that the above sum is bounded by
\begin{equation}
\begin{split}
    &\ll \sum_{f_1 \leq XV^{-1}} \sum_{d_2 \leq YV^{-1}} \sum_{d_3 f_3 \leq Z} ((\frac{X}{f_1})^{\frac{5}{6}}\frac{Y}{d_2} + \frac{X}{f_1} (\frac{Y}{d_2})^{\frac{5}{6}})(\log X Y)^\frac{7}{6}  \\ &\ll XYZV^{- \frac{1}{6}}(\log X \log Y)^{\frac{13}{6}} \log Z \ll XYZV^{- \frac{1}{6}} (\log X \log Y \log Z)^{\frac{13}{6}}.
    \label{S eqn 1}
\end{split}
\end{equation}

The second region $d_1,d_2,d_3,f_1,f_2,f_3 \leq V$ is trivially bounded by 
\begin{equation}
    V^6.
    \label{Six variables smaller than V}
\end{equation} 

We can bound the regions $d_i, d_j, f_i, f_j \leq V$ and $d_k, f_k > V$ for $\{i,j,k\}= \{1,2,3\}$ via trivial bounds by \begin{equation}
     \ll V^4(X \log X +Y \log Y + Z \log Z).
    \label{Three variables smaller than V}
\end{equation} 

There remain two regions to be bounded, $f_1,f_2,f_3 \leq V$ and $d_1, d_2,d_3 \leq V$. These are analogous but the first one is slightly more involved due to the presence of the characters $\chi_i$ so we will only explain the treatment of that one here. The relevant sum is
\begin{equation*}
    \sum_{ \substack{f_1,f_2,f_3 \leq V}}  \sum_{\substack{d_1 \leq \frac{X}{f_1} \\ d_2 \leq \frac{Y}{f_2} \\ d_3 \leq \frac{Z}{f_3}}} \frac{\mu(d_1 d_2 d_3 f_1 f_2 f_3)^2}{\tau(d_1 d_2 d_3 f_1 f_2 f_3)}  \chi_1(d_1)(\frac{f_2 f_3}{d_1})  \chi_2(d_2)(\frac{ f_1 f_3 }{d_2})\chi_3(d_3) (\frac{f_1 f_2}{d_3}).
\end{equation*}
We will now estimate the inner sum depending on the values of $f_1,f_2,f_3$. Note that since the terms are zero unless $f_1,f_2,f_3$ are pairwise coprime squarefree integers, $f_i f_j$ is a square for $i \neq j$ only if $1 = f_i = f_j$. The first case is $1 = f_1 = f_2 = f_3$, by applying $\sum_{\substack{n \leq x}} \mu(n)^2/\tau(n) \ll x (\log x)^{-\frac{1}{2}}$ three times to the sums over $d_1,d_2, d_3$ we see that the contribution of this part is \begin{equation}
    \ll XYZ (\log X \log Y \log Z)^{-\frac{1}{2}}.
    \label{S main term}
\end{equation} The second part is when exactly two of $f_1,f_2,f_3$ are equal to 1, we may assume that $ f_2 = f_3 = 1$ by symmetry. Now by applying Lemma \ref{S for short sums} to the sum over $d_2$ and using that $\chi_2 (\frac{f_1}{\_})$ is a non-principal character of conductor at most $4 |l| f_1$ we get a bound 
\begin{equation*}
\begin{split}
    &\ll_C |l|\sum_{f_1 \leq C} \frac{f_1}{\tau(f_1)} \sum_{\substack{d_1 \leq \frac{X}{f_1}}} \sum_{d_3 \leq Z} \frac{\mu(d_1 d_3 f_1 )^2}{\tau(d_1 d_3)} \tau(d_1 d_3)   Y (\log Y)^{-C}  \\ &\ll |l| XYZ\sum_{f_1 \leq C} \frac{1}{\tau(f_1)} \ll |l| V XYZ (\log Y)^{-C}
\end{split}
\end{equation*}
for all $C > 0$.
By instead applying Lemma \ref{S for short sums} to the sum over $d_3$ we get a similar bound with $(l,Y)$ and $(m,Z)$ switched. Bounding by their geometric mean and finding similar contributions for the other situations when two of $f_1, f_2, f_3$ are equal to 1 we get a total bound for this part of 
\begin{equation}
\begin{split}
        \ll_C VXYZ( \sqrt{\mid kl\mid } (\log X \log Y)^{- \frac{C}{2}} &+ \sqrt{\mid km\mid } (\log X \log Z)^{- \frac{C}{2}} \\ &+\sqrt{\mid lm\mid } (\log Y \log Z)^{- \frac{C}{2}}).
\end{split}
\label{S eqn 2}
\end{equation}
The last part is when none of the $f_1 f_2, f_1 f_3, f_2 f_3$ are equal to 1. We apply Lemma \ref{S for short sums} to the sum over $d_1$ where we use that $\chi_1(\frac{f_2 f_3}{\_})$ is a non-principal character of conductor at most $4|k| f_2 f_3$. For the other sums use trivial bounds to get
\begin{equation*}
\begin{split}
    &\ll_C |k| \sum_{ \substack{f_1,f_2,f_3 \leq V}} \frac{f_2 f_3}{\tau(f_1 f_2 f_3)} \sum_{\substack{d_2 \leq \frac{Y}{f_2} \\ d_3 \leq \frac{Z}{f_3}}} \frac{\mu(d_2 d_3 f_1 f_2 f_3)^2}{\tau(d_2 d_3)} \tau(f_1 d_2 d_3) X (\log X)^{-C} \\ &\ll V^3 |k|XYZ(\log X)^{-C}.
\end{split}
\end{equation*}
By applying Lemma \ref{S for short sums} instead to the sums over $d_2,d_3$ we get similar bounds so we may bound the sum by their geometric mean
\begin{equation}
    \ll_C |klm|^{\frac{1}{3}}V^3XYZ(\log X \log Y \log Z)^{-\frac{C}{3}}.
    \label{S eqn 3}
\end{equation}
We can now take for example $V = \frac{(\log X \log Y \log Z)^{20}}{\sqrt{\mid klm\mid }}$ and $C = 300$, we have assumed that $V \geq 1$ but if $\frac{(\log X \log Y \log Z)^{20}}{\sqrt{\mid klm\mid }} \leq 1$ then we can use the trivial inequality (\ref{S trivial inequality}) to find that 
\begin{equation*}
        S(X,Y,Z;k,l,m) \ll |klm|^{\frac{1}{4}}XYZ (\log X \log Y \log Z)^{-20}.
    \label{S when klm is large}
\end{equation*}
Using this choice we get the desired bound for \eqref{S eqn 1}, \eqref{Six variables smaller than V}, \eqref{S main term} and \eqref{S eqn 3}.
If we use the assumption \eqref{assumption for proof of S} we also get correct bounds for \eqref{Three variables smaller than V}, \eqref{S eqn 2}.
\end{proof}
\bibliographystyle{plain}
\bibliography{refs}
\end{document}